\documentclass[12pt,reqno]{amsart}
\usepackage{amssymb}
\usepackage{amsmath}

\usepackage{amscd}

\newcommand{\RNum}[1]{\uppercase\expandafter{\romannumeral #1\relax}}

\usepackage[T2A]{fontenc}
\usepackage[utf8]{inputenc}
\usepackage[russian,english]{babel}
\input{int.def}

\usepackage{tikz-cd}
\usetikzlibrary{cd}
\usepackage{dirtytalk}
\usepackage{hyperref}
\usepackage{xcolor}
\usepackage{centernot}

\usepackage{enumitem}

\numberwithin{equation}{section}

\def\fsw{\sf{FSW}}

\newcommand{\Poincare}{\text{Poincar\'e }}

\usepackage[mathcal]{euscript}

\usepackage{titlesec}
\titleformat{\section}[runin]{\bfseries}{\thesection.}{3pt}{}[.]

\myitemmargin 
\usepackage{geometry}
\newgeometry{vmargin={25mm}, hmargin={12mm,12mm}}   

\begin{document}

\title[Symplectic mapping class groups of blowups of tori]%
{Symplectic mapping class groups of blowups of tori}

\author{Gleb Smirnov}




\begin{abstract}
Let $\omega$ be a K{\"a}hler form on the real 
$4$-torus $T^4$. Suppose that $\omega$ satisfies an irrationality condition which can be achieved by an arbitrarily small perturbation of $\omega$. This note shows that the smoothly trivial symplectic mapping class group of the one-point symplectic blowup of $(T^4,\omega)$ is infinitely generated. 
\end{abstract}

\maketitle

\setcounter{section}{0}
\section{Main result}\label{main}
Let $(X,\omega)$ be a closed symplectic $4$-manifold. The smoothly trivial symplectic mapping class group 
$K(X,\omega)$ of $(X,\omega)$ is defined by:
\[
K(X,\omega) = \text{ker}\,\left[ \pi_0 \symp(X,\omega)  \to \pi_0 \diff(X)  \right].
\]
It is known that $K(X,\omega)$ is not finitely generated for some $K3$ surfaces. See \cite{Sher-Smith, Sm-2}. The aim of this note is to explore new examples of $4$-manifolds with $K(X,\omega)$ infinitely generated.
\medskip%

Let us consider $\rr^{4}$ endowed with a constant coefficient symplectic form $\omega_0$. Let $X = \rr^4/L$ be a real $4$-torus, $L$ being a lattice in $\rr^4$. Then $\omega_0$ descends to a symplectic form on $X$, which we denote by $\omega$. Set $\kappa = [\omega] \in H^2(X;\rr)$. We shall say that $\kappa$ is non-resonant if the following holds:
\begin{equation}\label{k-irrational}
\langle \kappa, a \rangle \neq 0\quad \text{for each $a \in H^2(X;\zz),\ a \neq 0,$}
\end{equation}
where $\langle \phantom{\cdot}, \phantom{\cdot} \rangle$ denotes the cup product pairing. Choose a translation-invariant $\omega$-compatible 
complex structure $J$ on $X$, so that $(X,J,\omega)$ is a K{\"a}hler manifold. Let $\pi \colon \tilde{X} \to X$ be the blowup of $X$ at a point. Choose $\lambda > 0$ small enough so that $\tilde{X}$ admits a K{\"a}hler form $\omega_{\lambda}$ in the cohomology class $\pi^{*}[\omega] - \lambda\,e$, where $e$ is the dual to the exceptional curve. (In fact, if \eqref{k-irrational} holds, then, for a suitably chosen $J$, the class $\pi^{*} [\omega] - \lambda\,e$ admits a K{\"a}hler representative for all $0 < \lambda < \sqrt{\langle \kappa, \kappa \rangle}$. See \cite{Lat-McD-Schl,Ent-Verbit-1}.) What we shall prove is:
\begin{theorem}\label{t:A}
If $\kappa$ is non-resonant, then 
$K(\tilde{X},\tilde{\omega}_{\lambda})$ is infinitely generated.
\end{theorem}
\smallskip%
The proof is similar to the proof of Theorem 1 in \cite{Sm-2}, but relies on different algebraic geometry from that of $K3$ surfaces; there are two steps:
\medskip%

\begin{enumerate}[leftmargin = 1cm, label=\Roman*.]
\item Consider the moduli space $\calm_{\kappa}$ of marked ($\kappa$-)polarized complex tori. It has the following properties: 
(a) $\calm_{\kappa}$ is a fine moduli space in the sense that there is a holomorphic fiber bundle $p \colon \calx \to \calm_{\kappa}$ 
whose fiber $X_t$ over $t \in \calm_{\kappa}$ is a complex 2-torus; (b) $\calm_{\kappa}$ is a contractible space; (c) each $X_t$ may be endowed with a K{\"a}hler form $\Omega_t$ such that $[\Omega_t] = \kappa$; and (d) $\calx$ has a canonical holomorphic cross-section $Z$.
\smallskip%

\noindent
Let $\tilde{\calx}$ denote the blowup of $\calx$ along $Z$. Composing the contraction of $Z$ with $p$, we obtain a fiber 
bundle $\tilde{p} \colon \tilde{\calx} \to \calm_{\kappa}$ whose fiber $\tilde{X}_t$
over $t \in \calm_{\kappa}$ is the one-point blowup of $X_t$.
\smallskip%

\noindent
Let $\calm_{\kappa,\lambda}$ be the subset of $\calm_{\kappa}$ defined as follows: $t \in \calm_{\kappa,\lambda}$ if $\tilde{X}_t$ carries a K{\"a}hler form $\tilde{\Omega}_t$ in the class
$[\tilde{\Omega}_t] = \kappa - \lambda\, e$. Let us choose a smooth-varying family of such K{\"a}hler forms $\tilde{\Omega}_t$, $t \in \calm_{\kappa,\lambda}$.
\smallskip%

\noindent
Consider the restriction $\tilde{p} \colon \tilde{\calx} \to \calm_{\kappa,\lambda}$. Identify 
$(\tilde{X}, \omega_{\lambda})$ with $(\tilde{X}_{t_0}, \tilde{\Omega}_{t_0})$ for some $t_0 \in \calm_{\kappa,\lambda}$. 
Since $\calm_{\kappa}$ is contractible, it follows that $\tilde{p} \colon \tilde{\calx} \to \calm_{\kappa}$ is $C^{\infty}$-trivial. Then $\tilde{p} \colon \tilde{\calx} \to \calm_{\kappa,\lambda}$ is also $C^{\infty}$-trivial. 
Thus we may view $\tilde{\Omega}_t$ as a family of symplectic forms on $(\tilde{X}, \omega_{\lambda})$ parametrized by $\calm_{\kappa,\lambda}$. Moser's theorem then gives a monodromy map
\[
\pi_1(\calm_{\kappa,\lambda}, t_0) \to K(\tilde{X}, \omega_{\lambda}).
\]
\item Let $\Delta_{\kappa,\lambda} \subset H^2(X;\zz)$ be the set of all classes $\delta$ that are indivisible (i.e., not a multiple) and that satisfy the following conditions:
\[
\langle \delta, \delta \rangle = 0,\quad 0 < \langle \kappa, \delta \rangle \leq \lambda.
\]
It will turn out that $\Delta_{\kappa,\lambda}$ is infinite, provided $\kappa$ is non-resonant. Following Kronheimer \cite{K}, we construct a homomorphism 
\[
q \colon K(\tilde{X},\tilde{\omega}_{\lambda}) \to \prod_{\delta \in \Delta_{\kappa,\lambda}} \zz_2,
\]
and then show that the composite homomorphism
\[
\pi_1(\calm_{\kappa,\lambda}, t_0) \xrightarrow{\text{\ Moser\ }} K(\tilde{X},\tilde{\omega}_{\lambda}) \xrightarrow{q} \prod_{\delta \in \Delta_{\kappa,\lambda}} \zz_2
\]
surjects onto $\bigoplus_{\delta \in \Delta_{\kappa,\lambda}} \zz_2 
\subset \prod_{\delta \in \Delta_{\kappa,\lambda}} \zz_2$.
\end{enumerate} 

One would hope to find a mirror symmetry proof of Theorem \ref{t:A} in the spirit of 
Sheridan-Smith \cite{Sher-Smith} and Hacking-Keating \cite{Hack-Keat}. The present note does not contain a clear hint towards such a proof.
\smallskip%

A recent study by Auroux and Smith \cite{Aur-Smith} provides the first, albeit implicit, information about $\pi_0 \symp(\tilde{X},\tilde{\omega}_{\lambda})$. Let us state their result as follows:
\begin{theorem}[Auroux-Smith, \cite{Aur-Smith}]
Let $C$ be a Riemann surface of genus $2$, and let it be equipped with a K{\"a}hler form $\tau$. Let $(\tilde{X},\tilde{\omega}_{\lambda})$ be as above. Let $Z$ be the product $(\Sigma \times \tilde{X}, \tau \oplus \tilde{\omega}_{\lambda})$. 
Choose a basis $\left\{ \eta_i \right\}$ in $H^1(Z;\zz)$, and choose $\epsilon > 0$ small so that $Z$ admits a K{\"a}hler form $\Omega_{\epsilon}$ in the cohomology class
\[
[\Omega_{\epsilon}] = [\tau \oplus \tilde{\omega}_{\lambda}] + \epsilon \sum_{i,j} c_{ij} \eta_t \wedge \eta_j\quad \text{for some $c_{ij} \in (0,1)$.}
\]
Suppose that $c_{ij}$ are independent over $\qq$; then there exists a surjective homomorphism of 
$K(Z,\Omega_{\epsilon})$ onto a free group of rank $N = N(\epsilon)$ which tends to infinity as $\epsilon \to 0$.
\end{theorem}
\smallskip%

There are other recent works on the symplectic mapping class groups in four dimensions: In \cite{Anj-Li-TJLi-Pins}, Anjos, J. Li, T.-J. Li, and Pinsonnault studied the change of the symplectic mapping class group (and of the higher homotopies of the symplectomorphism group) under deformations of the symplectic form and obtained specific stability results. In \cite{Li-TJLi-Wu}, J. Li, T.-J. Li, and Wu studied the symplectic mapping class groups of manifolds with $c_1(X) \cdot \omega > 0$. In particular, they found manifolds with trivial symplectic mapping class groups and also manifolds whose symplectic mapping class group is the pure braid group of $S^2$. In \cite{Bus-Li}, Buse and J. Li investigated the symplectic mapping class groups of a one-point blow-up of irrational ruled surfaces; they found new non-trivial symplectic mapping classes that are trivial smoothly. The works above focus on the symplectomorphism groups of rational ruled surfaces and their blow-ups. The methods of those works remain highly dependent on the rationality assumption, and it is unclear if they apply to the problem considered in the present paper.

\statebf Acknowledgements. 
\ An alternative proof of Lemma \ref{q-to-K} was found in a conversation with Jianfeng Lin; one that relies on the family blowup formula for SW-invariants rather than on negative inflation. I thank him for that and many other discussions. I also thank the referee for their helpful remarks. This work was supported by an SNSF Ambizione fellowship.

\section{Negative inflation}\label{negativ} The notation are as in \S\,\ref{main}. 
Let $S_{\kappa, \lambda}$ be the space of those symplectic forms on $\tilde{X}$ 
which can be joined with 
$\tilde{\omega}_{\lambda}$ 
through a path of cohomologous symplectic forms. Following Kronheimer \cite{K} and McDuff \cite{McD-1}, we consider 
the following fibration:
\begin{equation}\label{mcduff-fibr}
\symp(\tilde{X},\omega_{\lambda}) \to \diff(\tilde{X}) \xrightarrow{\psi \to (\psi^{-1})^{*} \tilde{\omega}_{\lambda} } S_{\kappa,\lambda},
\end{equation}
and its associated long exact sequence:
\[
\cdots \to \pi_1 \diff(\tilde{X}) \to 
\pi_1 (S_{\kappa,\lambda}) \to K(\tilde{X},\omega_{\lambda}).
\]
The following statement is due to McDuff and can be found in \cite{McD-1}, see Lemma 2.1, Lemma 2.2, and Corollary 2.3.
\begin{theorem}[McDuff, \cite{McD-1}]\label{mcduff-alpha}
There are inclusions 
$\alpha_{\mu} \colon S_{\kappa,\lambda} \to S_{\kappa,\,\lambda - \mu}$ 
for $0 \leq \mu < \lambda$ such that $\alpha_0$ is the identity, 
$\alpha_{\mu_1 + \mu_2}$ is homotopic to $\alpha_{\mu_2} \circ \alpha_{\mu_1}$ whenever all three are defined, and such that the following diagram is 
commutative:
\[
\begin{tikzcd}
\pi_1 \diff(\tilde{X}) \arrow{d}{\text{\normalfont{identity}}} \arrow{r}{} & \pi_1 (S_{\kappa,\lambda}) \arrow{d}{\alpha_{\mu *}}\\
\pi_1 \diff(\tilde{X}) \arrow{r}{} & \pi_1 (S_{\kappa,\,\lambda - \mu})\,,
\end{tikzcd}
\]
where the horizontal maps are as in \eqref{mcduff-fibr}.
\end{theorem}
\begin{proof}
Let $\calj_{\lambda - \mu}$ denote the space of almost-complex structures that are tamed by some form in $S_{\kappa,\,\lambda - \mu}$
\footnote{Recall that $\omega$ tames $J$ if $\omega(v,Jv) > 0$ for all $v \neq 0$.}. 
McDuff shows (see \cite[Lemma 2.1]{McD-1}) that $\calj_{\lambda - \mu}$ is homotopy equivalent to $S_{\kappa,\,\lambda - \mu}$. The is proof is as follows: 
Let $\calp_{\kappa,\,\lambda - \mu}$ be the space of pairs
\[
\calp_{\kappa,\,\lambda - \mu} = 
\left\{ 
(\omega,J) \in S_{\kappa,\,\lambda - \mu} \times \calj_{\lambda - \mu}\ |\ 
\text{$\omega$ tames $J$}.
\right\}
\]
Both projections 
$\calp_{\kappa,\,\lambda - \mu} \to S_{\kappa,\,\lambda - \mu}$, 
$\calp_{\kappa,\,\lambda - \mu} \to \calj_{\kappa,\,\lambda - \mu}$ are fibrations 
with contractible fibers, and hence are homotopy equivalences. We will show below that 
$\calj_{\kappa,\,\lambda} \subset \calj_{\kappa,\,\lambda - \mu}$. 
Granted this, $\alpha_{\mu}$ are induced from the inclusion 
$\calj_{\kappa,\,\lambda} \subset \calj_{\kappa,\,\lambda - \mu}$ using 
the homotopy equivalences above. All the claimed properties of $\alpha_{\mu}$ are rather straightforward form the definition. \qed
\end{proof}
\smallskip%

To complete the proof of Theorem \ref{mcduff-alpha}, we need 
the following statement:
\begin{theorem}[Buse, \cite{Bus}]\label{deflation}
Fix a symplectic $4$-manifold $(X,\omega_0)$ and any $\omega_0$-tamed 
almost-complex structure $J$. Suppose that $X$ admits an embedded $J$-holomorphic curve 
$C \subset X$ of self-intersection $(-m)$; then, for each $\mu$ such that 
\[
0 \leq \mu < \dfrac{1}{m} \int_{C} \omega_0,
\]
there are symplectic forms 
$\omega_{\mu}$, all taming $J$, which satisfy 
\[
[\omega_{\mu}] = [\omega_{0}] + \mu [C],
\]
where $[C] \in H^2(X;\rr)$ is \Poincare dual to $C$. 
\end{theorem}
\begin{proof}
The proof is found in \cite[\S\,2]{Bus}. \qed
\end{proof}
\begin{lemma}
$\calj_{\kappa,\,\lambda} \subset \calj_{\kappa,\,\lambda - \mu}$ for all 
$0 \leq \mu < \lambda$.
\end{lemma}
\begin{proof}
For each $J \in \calj_{\lambda}$, $\tilde{X}$ carries a smooth $J$-holomorphic sphere 
of self-intersection number $(-1)$. Let $C$ be that sphere and apply Theorem \ref{deflation}. \qed
\end{proof}
\smallskip%

This finishes the proof of Theorem \ref{mcduff-alpha}. 
\section{Family Seiberg-Witten invariants}\label{sw-family}
In this section, we recall the basic properties of Seiberg-Witten invariants for families, specializing to the case of $b^{+} = 3$. We refer to \cite{Morg, Nic} for a comprehensive introduction to ordinary Seiberg-Witten invariants and to \cite{LL} for a detailed account of their family version. Some recent applications of family SW-invariants can be found in \cite{Kr-Mr, Barag-2, Lin-1}.
\medskip%

Let $X$ be a closed, connected, smooth oriented $4$-manifold. Let $\calx \to B$ be a fiber bundle whose fiber is $X$ and whose base $B$ is a smooth oriented manifold. In this section we insist that $B$ is either the $2$-disk $D$ or the $2$-sphere $S^2$ and also that $b^{+}_2(X) = 3$. 
\smallskip%

Let $T_{\calx/B}$ denote the vertical tangent bundle of $\calx$. Pick a metric on $T_{\calx/B}$, which we regard as a family of fiberwise metrics $\left\{ g_b \right\}_{b \in B}$. Pick also a spin$^\cc$ structure $\fr{s}$ on $T_{\calx/B}$ of $\calx$. Let $\fr{s}_b$ denote the restriction of $\fr{s}$ to a fiber $X_b$, $b \in B$. Hereafter, given any object 
on $\calx$, the object with subscript $b$ denotes the restriction 
of the given object to the fiber $X_b$. 
\smallskip%

Associated to $\fr{s}$, there are 
spinor bundles $W^{\pm} \to B$ and determinant line bundle $\call$, which we regard as 
families of bundles 
\[
W^{\pm} = \bigcup_{b \in B} W^{\pm}_b,\quad \call = \bigcup_{b \in B} \call_{b}.
\]
We let $c_1(\fr{s}) \in H^2(\calx;\zz)$ and 
$c_1(\fr{s}_b) \in H^2(X_b;\zz)$ denote the Chern classes of $\fr{s}$ and 
$\fr{s}_b$ respectively;
\[
c_1(\fr{s}) = c_1(\call) \in H^2(\calx;\zz),
\quad c_1(\fr{s}_b) = c_1(\call_b) \in H^2(X_b;\zz).
\]
Consider the configuration space
\[
\calc = \left\{ (b, \varphi_{b}, A_{b})\ |\ 
b \in B,\ 
\text{$\varphi_b$ is a section of $W^{\pm}_b$,}\ 
\text{$A_b$ is a $\UU(1)$-connections on $\call_b$} 
\right\},
\]
and let $\calc_b$ be the fiber of $\calc \to B$ over $b \in B$. Choose a family $\left\{ \eta_b \right\}_{b \in B}$ of $g_b$-self-dual 2-forms. The ($\eta_b$-)perturbed Seiberg-Witten equations for a triple 
$(b, \varphi_b, A_b) \in \calc$ read: 
\begin{equation}\label{eq:sw}
 \begin{cases}
   \cald_{A_b} \varphi_b = 0, 
   \\
   F^{+}_{A_{b}}  = \sigma(\varphi_b) + i\,\eta_b.
 \end{cases}
\end{equation}
Here $\sigma(\varphi)$ is the squaring map, and $F^{+}_{A_{b}}$ is the self-dual part of $F_{A_{b}}$, the curvature of the connection $A_{b}$, and $\cald_{A_b} \colon \Gamma(W_{b}^{+}) \to \Gamma(W_{b}^{-})$ is the Dirac operator.
\smallskip%

A solution $(b, \varphi_{b}, \cala_{b})$ to equations \eqref{eq:sw} is said to be reducible if $\varphi_b = 0$. Let us give a sufficient condition on $(g_b,\eta_b)$ which ensures that \eqref{eq:sw} has 
no reducible solutions. Setting $\varphi_b = 0$, the equations take the form:
\[
F^{+}_{A_{b}}  = i\,\eta_b.
\]
The latter equation has a solution $A_b$ if and only if
\[
\langle \eta_b \rangle_{g_b} + 2 \pi \langle c_1(\fr{s}_b) \rangle_{g_b} = 0,
\]
where $\langle \phantom{\eta} \rangle_{g_b}$ stands for the self-dual part of the harmonic representative of the $2$-form in question, and, as before,  
$c_1(\fr{s}_b)$ is the Chern class of $\fr{s}_b$;
\[
[F_{A_{b}}] = -2 \pi i  c_1(\fr{s}_b) \in H^2(X_b;\rr).
\]
Suppose that $\left\{ g_b,\eta_b \right\}_{b \in B}$ satisfy 
\begin{equation}\label{eq:irreducible}
\langle \eta_b \rangle_{g_b} + 2 \pi \langle c_1(\fr{s}_b) \rangle_{g_b} \neq 0\quad \text{for all $b \in B$,}
\end{equation}
then \eqref{eq:sw} has no reducible solutions. With \eqref{eq:irreducible} assumed, we consider the family moduli space
\[
\mathfrak{M}^{\fr{s}}_{(g_b,\eta_b)} = 
\left\{ (b,\varphi_b,A_b) \in \calc\ |\ (b,\varphi_b,A_b)\ \text{solves}\ \eqref{eq:sw}\right\}/\left\{ \text{gauge transformations} \right\}.
\]
By the Sard-Smale theorem, for a generic $\left\{ g_b,\eta_b \right\}_{b \in B}$, this moduli space is either empty or is a compact smooth manifold of dimension 
\begin{equation}\label{expect-d}
d(\mathfrak{s},B) = 
\frac{1}{4}( c_1^2(\fr{s}_b) - 3 \sigma(X) -2 \chi(X) ) + \text{dim}\,B.
\end{equation}
Although the family invariants may be defined in general, we shall focus 
on the case of $d(\mathfrak{s},B) = 0$. Define the mod-2 Seiberg-Witten invariant for the spin$^\cc$ structure $\fr{s}$ and the family $\left\{ (g_b,\eta_b) \right\}_{b \in B}$ by setting
\[
\fsw_{(g_b,\eta_b)}(\fr{s}) = \text{Cardinality}\left( \mathfrak{M}^{\fr{s}}_{(g_b,\eta_b)} \right)\,\text{mod}\,2.
\]
The number depends on the choice of $\left\{ g_b,\eta_b \right\}_{b \in B}$. However, we can remove this dependence by restricting ourselves 
to the families with vanishing winding number, which we now define. Observe that $\left\{ H^2(X_b;\zz) \right\}_{b \in B}$ are all canonically isomorphic to each other, and are isomorphic to an abstract lattice $E \cong H^2(X;\zz)$. Such an isomorphism is indeed exists, because $B$ is simply-connected. Let us consider the (punctured) positive cone $\mathbf{K}$;
\[
\mathbf{K} = \left\{ 
a \in E \otimes \rr \,|\, a^2 > 0 \right\}.
\]
The space $\mathbf{K}$ is homotopy equivalent to 
$S^{b_2^{+}(X) - 1}$, which in our case would be $S^2$. Assuming that $\left\{ g_b,\eta_b \right\}_{b \in B}$ satisfy \eqref{eq:irreducible}, 
we have a map
\[
B \to \mathbf{K}\quad \text{defined by}\ b \to 
[
\langle \eta_b \rangle_{g_b} + 2 \pi \langle c_1(\call_b) \rangle_{g_b} 
] \in H^2(X_b;\rr),
\]
and followed by the identification of $H^2(X_b;\rr)$ with $E \otimes \rr$. In the case of $B \cong S^2$, this map has a degree, which we call the winding number of $\left\{ g_b,\eta_b \right\}_{b \in B}$. 
\begin{lemma}\label{fsw-diff}
Suppose that we are given two families 
$\left\{ g_b,\eta_b \right\}_{b \in B}$ and 
$\left\{ g'_b,\eta'_b \right\}_{b \in B}$, each having vanishing winding number; then 
\[
\fsw_{(g_b,\eta_b)}(\fr{s}) = \fsw_{(g'_b,\eta'_b)}(\fr{s}).
\]
\end{lemma}
\begin{proof}
This is essentially proved in \cite[Prop.\,3.6]{LL}. 
See also Lemma 2.1 in \cite{Barag-2}. \qed
\end{proof}
\smallskip%

Thus, the 
family SW-invariant, if being defined for a family with vanishing winding number, depends only on the (smooth) topology of $\calx \to B$ but not on our choice of $g_b$ and $\eta_b$. 
\begin{lemma}\label{it-will-wind}
Suppose that a family $\left\{ g_b,\eta_b \right\}_{b \in B}$, a spin$^\cc$ structure $\fr{s}$, and the conjugate spin$^\cc$ structure, written as $-\fr{s}$, are so that $\left\{ g_b,\eta_b \right\}_{b \in B}$ has vanishing winding number w.r.t. to both $\fr{s}$ and $-\fr{s}$; then 
\[
\fsw_{(g_b,\eta_b)}(\fr{s}) = \fsw_{(g_b,\eta_b)}(-\fr{s}).
\]
\end{lemma}
\begin{proof}
See, e.g., \cite[Prop.\,2.8]{Barag-2}. See also \cite[\S\,3]{Sm-2} and Proposition 2.2.22 in \cite{Nic}. \qed
\end{proof}
\smallskip%

One can also weaken the dependence of a family invariant on the underlying spin$^\cc$ structure.
\begin{lemma}\label{wont-matter}
Suppose that we are given spin$^\cc$ structures 
$\fr{s}$ and $\fr{s}'$ such that the restrictions 
$\fr{s}_b$ and $\fr{s}'_b$ have equal Chern classes;
\[
c_1(\fr{s}_b) = c_1(\fr{s}'_b) \in H^2(X_b;\zz).
\]
If $H^2(X;\zz)$ is torsion-free, then
\[
\fsw_{(g_b,\eta_b)}(\fr{s}) = \fsw_{(g_b,\eta_b)}(\fr{s}').
\]
\end{lemma}
\begin{proof}
This is well known, so we just sketch the proof.
For each $b \in B$, we let $W_b^{\pm}$ and $W_b^{\pm'}$ be the spinor bundles for $\fr{s}_b$ and $\fr{s}'_b$ respectively. Let $\calc_b$ and $\calc'_b$ be the configuration spaces for $\fr{s}_b$ and $\fr{s}'_b$ respectively. Since $H^2(X_b;\zz)$ has no $2$-torsion, it follows from Gompf's theorem \cite[Prop.\,1]{Gompf} that $\fr{s}_b$ and $\fr{s}'_b$ are isomorphic: that is, there is a unitary isomorphism $w_b \colon W_{b}^{\pm} \to W_{b}^{\pm'}$ which preserves the Clifford multiplication. Furthermore, modulo gauge transformations, $w_b$ is unique. Associated to $w_b$, there is an isomorphism $\calc_b \to \calc_b'$, which we also denote by $w_b$. This isomorphism is gauge-equivariant and it takes SW-solutions to SW-solutions. $w_b$ descends to an isomorphism $\hat{w}_b \colon \calc_b/\left\{ \text{gauge transformations} \right\} \to  \calc'_b/\left\{ \text{gauge transformations} \right\}$. Although $w_b$ is only defined up to gauge transformations, the map $\hat{w}_b$ is unique. Thus we can use $\hat{w}_b$ to identify the moduli spaces 
$\mathfrak{M}^{\fr{s}}_{(g_b,\eta_b)}$ and 
$\mathfrak{M}^{\fr{s}'}_{(g_b,\eta_b)}$, even though there is no isomorphism between the spaces of SW-solutions for $\fr{s}$ and $\fr{s}'$. This completes the proof. For a proof in greater generality, see \cite[\S\,2.2]{Barag-1}. See also Lemma 2.7 in \cite{Lin-2}. \qed
\end{proof}

\section{Seiberg-Witten for symplectic manifolds}
The following facts are well-known; see, e.g., \cite[\S\,3.3]{Nic}, \cite[Ch.\,7]{Morg} for details. Let $(X,\omega)$ be a closed symplectic $4$-manifold, $J$ be an $\omega$-compatible almost-complex structure, 
and $g(\cdot,\cdot) = \omega(\cdot,J\cdot)$ be the associated Hermitian metric. Denote by $\kappa$ the class $[\omega] \in H^2(X;\rr)$. The canonical spin$^{\cc}$ structure $\fr{s}$ has bundles 
\[
W^{+} =  \Lambda^{0,0} \oplus \Lambda^{0,2},\quad 
W^{-} = \Lambda^{0,1},
\]
and determinant line bundle $K^{*}_{X}$. There is a special connection $A_0$ on $K^{*}_X$ called the Chern connection such that the induced Dirac operator is
\[
\cald_{A_0} \colon \Gamma(\Lambda^{0, 0} \oplus \Lambda^{0, 2}) \to 
\Gamma(\Lambda^{0,1}),\quad \cald_{A_0} = \bar{\del} + \bar{\del}^{*}\quad\text{up to a positive multiplier.}
\]
See \cite[Ch.\,3]{Morg} and Proposition 1.4.23 in \cite{Nic}.
\smallskip%

Let $\varepsilon \in H^2(X;\zz)$ and let $L_{\varepsilon}$ be a complex line bundle on $X$ with $c_1(L_{\varepsilon}) = \varepsilon$. Let $\fr{s}_{\varepsilon}$ be the spin$^{\cc}$ structure on $X$ given by
\begin{equation}\label{spin-eps}
W^{+} = L_{\varepsilon} \oplus \left( \Lambda^{0,2} \otimes L_{\varepsilon} \right),\quad 
W^{-} = \Lambda^{0,1} \otimes L_{\varepsilon}.
\end{equation}
The determinant line bundle corresponding to $\fr{s}_{\varepsilon}$ is 
$K^{*}_X \otimes L_{\varepsilon}^2$. A choice of a $\mib{U}(1)$-connection on $L_{\varepsilon}$ combines with $A_0$ to give a connection 
$A = A_0 + 2\,B$ on $K^{*}_X \otimes L_{\varepsilon}^2$, and all $\mib{U}(1)$-connections on $K^{*}_X \otimes L_{\varepsilon}^2$ have such a form. 
\smallskip%

Write $\varphi = (\ell,\beta)$ for $\varphi \in W^{+}$. 
The Seiberg-Witten equations, with Taubes' perturbing term $\eta$;
\begin{equation}\label{eq:taubes-eta}
    i\,\eta = F^{+}_{A_0} - i\, \rho\, \omega,
\end{equation}
are as follows:
\begin{equation}\label{eq:sw-eq-sympl}
 \begin{cases}
   \bar{\del}_{B} \ell + \bar{\del}_{B}^{*} \beta = 0, 
   \\
   2\,F^{0,2}_{B} = \dfrac{\ell^{*} \beta}{2},
   \\
   2 (F^{+}_{B})^{1,1} = \dfrac{i}{4}\,(|\ell|^2 - |\beta|^2 - 4\,\rho) \omega.
 \end{cases}
\end{equation}
In this symplectic case, \eqref{expect-d} takes the form:
\begin{equation}\label{expect-d-sympl}
d(\fr{s}_{\varepsilon},B) = \langle K^{*}_{X}, 
\varepsilon \rangle + \langle \varepsilon, \varepsilon \rangle + \text{dim}\,B.
\end{equation}
For the Chern class of $\fr{s}_{\varepsilon}$, we have that: 
\[
c_1(\fr{s}_{\varepsilon}) = K^{*}_{X} + 2\,\varepsilon;
\]
for the conjugate spin$^{\cc}$ structure $-\fr{s}_{\varepsilon}$, we have that:
\begin{equation}\label{s-conj}
-\fr{s}_{\varepsilon} = \fr{s}_{K_{X} - \varepsilon},
\end{equation}
since if $K_{X}^{*} + 2 \varepsilon = - (K_{X}^{*} + 2 \varepsilon')$, then $\varepsilon' = K_{X} - \varepsilon$.
\begin{lemma}\label{l:no-reducibles}
Equations \eqref{eq:sw-eq-sympl} have no reducible solutions for 
$\rho > 0$ large enough.
\end{lemma}
\begin{proof}
Reducible solutions exist only if
\[
\langle \eta \rangle_{g} + 2 \pi \langle c_1(\fr{s}_{\varepsilon}) \rangle_{g} = 0.
\]
Substituting \eqref{eq:taubes-eta} for $\eta$ and multiplying both sides on $\omega$ gives:
\[
\langle -i\,F^{+}_{A_0} - \rho\,\omega \rangle \wedge \omega + 
2\,\pi\,\langle c_1(\fr{s}_{\varepsilon}) \rangle \wedge \omega 
= - \rho\, \omega \wedge \omega + \ldots = 0.
\]
But integrating over $X$ gives: 
\[
 \int_X - \rho\, \omega \wedge \omega + \ldots < 0\quad 
 \text{for $\rho > 0$ sufficiently large. }
\]
This finishes the proof. \qed
\end{proof}
\medskip%

In the proof of Theorem \ref{t:A}, we will use the following 
two results:
\begin{theorem}[Taubes, \cite{Taub-2}]\label{thm:taubes}
Suppose that 
\[
\varepsilon \neq 0\quad\text{and}\quad \langle \varepsilon, \kappa \rangle \leq 0.
\]
Then equations \eqref{eq:sw-eq-sympl} have no solutions for 
$\rho > 0$ large enough.
\end{theorem}
\begin{proof}
See inequality (3.3.23) in \cite[Theorem 3.3.29]{Nic} and the discussion afterward. It implies that $\text{deg}_{\omega}(L_{\varepsilon}) = \langle \varepsilon, \kappa \rangle$ is not negative in the presence of a solution. Moreover, it shows that, in the presence of a solution, $\text{deg}_{\omega}(L_{\varepsilon}) = 0$ only if $\varepsilon = 0$. \qed
\end{proof}
\medskip%

If $(X,\omega)$ is K{\"a}hler, then we have the following result: 
\begin{theorem}\label{t:kahler}
If $\varepsilon \centernot\in H^{1,1}(X;\rr)$, then the equations \eqref{eq:sw-eq-sympl} have no solutions. If $\varepsilon \in H^{1,1}(X;\rr)$ and $\rho > 0$ is large enough, then 
solutions to \eqref{eq:sw-eq-sympl} are irreducible and, modulo gauge transformations, are in one-to-one correspondence with the set of effective
divisors in the class $\varepsilon$.
\end{theorem}
\begin{proof}
See \cite[Ch.\,7]{Morg}. \qed
\end{proof}

\section{Kronheimer's homomorphism}\label{K-map}
Let $(X,\omega)$ be a closed symplectic $4$-manifold and let again $S_{[\omega]}$ be the space of symplectic forms deformation equivalent of 
$\omega$ within forms of constant cohomology class. Pick a class $\varepsilon \in H^2(X;\zz)$ such that: 
\[
\langle K^{*}_{X}, 
\varepsilon \rangle + \langle \varepsilon, \varepsilon \rangle = -2,\quad 
\langle [\omega], \varepsilon \rangle \leq 0. 
\]
In particular, we have:
\[
d(\fr{s}_{\varepsilon},B) = \text{dim}\,B - 2.
\]
Here we sketch the construction of Kronheimer's homomorphism \cite{K}:
\[
Q_{\varepsilon} \colon \pi_1 (S_{[\omega]}) \to \zz_2.
\]
Crucially, we need the $b^{+} = 3$ version studied in \cite{Sm-2}.
\medskip%

Let $\left\{ \omega_t \right\}_{t \in S^1}$ be a loop in $S_{[\omega]}$ based at $\omega$. Choose a family $\left\{ J_t \right\}_{t \in S^1}$ of $\omega_t$-compatible almost-complex structures on $X$. Consider the associated family of Hermitian metrics $\left\{ g_t \right\}_{t \in S^1}$ on $X$. Let $D$ denote the 2-disk. Choose a family 
$\left\{ g_b \right\}_{b \in D}$ of Riemannian metrics in such way that it provides a nullhomotopy of the family $\left\{ g_t \right\}_{t \in S^1}$. We regard $\left\{ g_b \right\}_{b \in D}$ as a family of fiberwise metrics on a trivial bundle $\calx = X \times D$.
\smallskip%

Let $\fr{s}_{\varepsilon}$ be the spin$^{\cc}$ structure on $X$ given 
by \eqref{spin-eps}. Extend $\fr{s}_{\varepsilon}$ to a spin$^{\cc}$ structure on the whole of $T_{\calx/D}$. We shall continue to use $\fr{s}_{\varepsilon}$ to denote this spin$^{\cc}$ structure. Letting ${A_0}_t$ be 
the Chern connection on $K^{*}_X$ determined by $g_t$, we set:
\begin{equation}\label{eta-t-gamma}
\eta_{t} = -i F_{{A_0}_t}^{+} - \rho\, \omega_t.
\end{equation}
We choose $\rho > 0$ large enough so that 
\[
\int_X \langle \eta_t \rangle_{g_t} \wedge \omega + 
2 \pi \langle c_1(\fr{s}_{\varepsilon}) \rangle_{g_t} \wedge \omega < 0\quad \text{for all $t \in S^1$.}
\]
Let $\left\{ \eta_b \right\}_{b \in D}$ be a family of fiberwise 
$g_b$-self-dual forms on $\calx$ that agree with $\eta_t$ 
on $\del D$ and such that
\begin{equation}\label{admin}
\int_X \langle \eta_b \rangle_{g_b} \wedge \omega + 
2 \pi \langle c_1(\fr{s}_{\varepsilon}) \rangle_{g_b} \wedge \omega < 0\quad \text{for all $b \in D$.}
\end{equation} 
In \cite{Sm-2}, $\eta_b$ was called an admissible extension of $\eta_t$. Consider the Seiberg-Witten equations parametrized by the 
family $\left\{ (g_b, \eta_b) \right\}_{b \in D}$. 
It follows from \eqref{admin} that, for all $b \in B$, these equations have no reducible solutions. Choosing $\rho$ larger if necessary, we apply Theorem \ref{thm:taubes} to ensure the absence of solution for all 
$(g_t,\eta_t)$, $t \in S^1$. 
Now the relative Sard-Smale theorem applies: 
Perturbing $\left\{ \eta_b \right\}_{b \in D}$ relative to 
$\left\{ \eta_t \right\}_{t \in D}$, we 
ensure $\fr{M}^{\fr{s}_{\varepsilon}}_{(g_b,\eta_b)}$ is a manifold of 
dimension $d(\fr{s}_{\varepsilon},D) = 0$. Now set:
\[
Q_{\varepsilon}(\left\{ \omega_t \right\}_{t \in S^1}) = \text{Cardinality}\left( \mathfrak{M}^{\fr{s}_{\varepsilon}}_{(g_b,\eta_b)} \right)\,\text{mod}\,2.
\]
One shows that the number $Q_{\varepsilon}(\left\{ \omega_t \right\}_{t \in S^1})$ does not depend on our choice of $\left\{ \eta_b \right\}_{b \in D}$ provided that \eqref{admin} holds; nor it depends on a particular choice of $\left\{ \omega_t \right\}_{t \in S^1}$ within a (free) homotopy class. The map $Q_{\varepsilon}$ thus gives a group homomorphism $\pi_1(S_{[\omega]}) \to \zz_2$. See \cite{K} for details.

\section{The homomorphism $q$}\label{q-map}
Let $(X,\omega)$ be a symplectic $4$-torus and set $\kappa = [\omega] \in H^2(X;\rr)$. Denote by $(\tilde{X},\omega_{\lambda})$ the symplectic $\lambda$-blowup of $(X,\omega)$. Then, letting $e \in H^2(\tilde{X};\zz)$ be the 
class of the exceptional $(-1)$-curve, 
\[
[\omega_{\lambda}] = \kappa - \lambda\, e.
\]
Let $S_{\kappa,\lambda}$, $K(\tilde{X},\omega_{\lambda})$ be as in \S\,\ref{negativ}. We saw them fit the exact sequence: 
\[
\cdots \to \pi_1 \diff(\tilde{X}) \to 
\pi_1 (S_{\kappa,\lambda}) \to K(\tilde{X},\omega_{\lambda}).
\]
Pick a class $\delta \in H^2(X;\zz)$ such that:
\begin{equation}\label{a's}
\langle \delta, \delta \rangle = 0,\quad 
0 < \langle \kappa, \delta \rangle \leq \lambda.
\end{equation}
In particular, we have:
\[
d(\fr{s}_{\delta - e},B) = \text{dim}\,B - 2,\quad 
\langle \kappa - \lambda\, e, \delta - e \rangle \leq 0.
\]
Let $q_{\delta - e} \colon \pi_1 (S_{\kappa,\lambda}) \to \zz_2$ be defined by 
\[
q_{\delta - e} = Q^{+}_{\delta - e} + Q^{-}_{\delta - e},
\]
where $Q^{+}_{\delta - e} = Q_{\delta - e}$ is Kronheimer's homomorphism for $\varepsilon = \delta - e$ and $Q^{-}_{\delta - e}$ is defined as follows: 
Choose $\mu > 0$ so close to $\lambda$ that 
\begin{equation}\label{mu-close-eps}
\langle 2\,e - \delta, \kappa - (\lambda - \mu)\,e \rangle \leq 0.
\end{equation}
Then define $Q^{-}_{\delta - e}$ as the following composition:
\[
\pi_1 (S_{\kappa,\lambda}) 
\xrightarrow{\alpha_{\mu*}} 
\pi_1 (S_{\kappa,\lambda - \mu}) 
\xrightarrow{Q_{2\,e - \delta}} 
\zz_2,
\]
where $\alpha_{\mu}$ is as in Theorem \ref{mcduff-alpha}, and 
$Q_{2\,e - \delta}$ is Kronheimer's homomorphism for $\varepsilon = 2\,e - \delta$. The latter is indeed well defined because of \eqref{mu-close-eps} and also because $d(\fr{s}_{2\,e - \delta}, D) = 0$. 
\begin{lemma}\label{q-to-K}
The composite homomorphism
\[
\pi_1 \diff(\tilde{X}) \to \pi_1 (S_{\kappa,\lambda}) \xrightarrow{q_{\delta - e}} \zz_2
\]
is a nullhomomorphism.
\end{lemma}
\begin{proof} 
Suppose that there is a family of symplectomorphisms 
$f_t \colon (\tilde{X},\omega_t) \to (\tilde{X},\omega_{\lambda})$, $t \in \del D$. Then, via 
the clutching construction, $\left\{ f_t \right\}_{t \in \del D}$ gives rise to a fiber bundle:
\[
\caly \to B,\qquad B = D/\del D,\quad \caly = \calx \cup \tilde{X}/\sim,\ \text{where $(t,x) \sim f_t(x)$ for each $t \in \del D$ and $x \in \tilde{X}$.}
\]
Let $J$ be an $\omega$-compatible almost-complex structure on $\tilde{X}$, and let $g$ be the associated Hermitian metric. Set: $J_t = (f_t^{-1})_* \circ J\circ (f_t)_*$, $g_t = g \circ (f_t)_{*}$. There is a $g$-self-dual form $\eta$ on $\tilde{X}$ such that for each $t \in \del D$, 
\[
(f^{-1}_t)^{*} \eta_t = \eta, \quad 
\text{where $\eta_t$ is as in \eqref{eta-t-gamma}.}
\]
Choose a family $\left\{ g_b \right\}_{b \in D}$ of 
Riemannian metrics on $\tilde{X}$ that agree with $\left\{ g_t \right\}_{t \in \del D}$ at each $t \in \del D$, and choose then a family of $g_b$-self-dual forms $\left\{ \eta_b \right\}_{b \in D}$ 
that agree with $\left\{ \eta_t \right\}_{t \in \del D}$ at each $t \in \del D$. 
We may arrange that $\left\{ \eta_b \right\}_{b \in D}$ satisfies condition \eqref{admin}, i.e., 
\[
\int_X \langle \eta_b \rangle_{g_b} \wedge \omega + 
2 \pi \langle c_1(\fr{s}_{\varepsilon}) \rangle_{g_b} \wedge \omega < 0\quad \text{for all $b \in D$.}
\]
In short, this is done as follows: With $\eta_t$ as in \eqref{eta-t-gamma}, we choose $\rho$ large enough so that \eqref{admin} holds for all $t \in \del D$. For each $g_b$, the space of self-dual forms $\eta_b$ satisfying \eqref{admin} is non-empty and contractible. Therefore $\eta_t$, $t \in \del D$ has an extension $\eta_b$, $b \in D$ with the required property. See \cite[\S\,5]{Sm-2} for details. 
\smallskip%

Observe that we also get 
a family $\left\{ (g_b,\eta_b) \right\}_{b \in B}$ on $\caly$. This family has vanishing winding number, since $\eta_b$ satisfies \eqref{admin}. Choose an extension of $\fr{s}_{\delta - e}$ from $T_{\calx/D}$ to $T_{\caly/B}$. (We note that such an extension is not unique, since $H^2(B;\zz) \neq 0$; however, by Lemma \ref{wont-matter}, the family SW-invariants are independent of the choice of extension.) From definition of $Q_{\delta - e}$, we have
\[
Q_{\delta - e}(\left\{ \omega_t \right\}_{t \in S^1}) = 
\fsw_{(g_b, \eta_b)}(\fr{s}_{\delta - e}).
\]
On the other hand, if we apply $\alpha_{\mu}$, then we get:
\[
Q_{2\,e - \delta}(\left\{ \alpha_{\mu}(\omega_t) \right\}_{t \in S^1}) = 
\fsw_{(g'_b, \eta'_b)}(\fr{s}_{2\,e - \delta}),
\]
where $(g'_b, \eta'_b)$ are defined 
similarly to $(g_b, \eta_b)$, but using the loop $\alpha_{\mu}(\omega_t)$ instead of $\omega_t$. Because of Lemma \ref{fsw-diff},
\[
\fsw_{(g'_b, \eta'_b)}(\fr{s}_{\delta - e}) = 
\fsw_{(g_b, \eta_b)}(\fr{s}_{\delta - e}),
\]
and then, using Lemma \ref{it-will-wind} and \eqref{s-conj}, 
\[
\fsw_{(g'_b, \eta'_b)}(\fr{s}_{2\,e - \delta}) = 
\fsw_{(g'_b, \eta'_b)}(-\fr{s}_{2\,e - \delta}) \overset{\eqref{s-conj}}{=} \fsw_{(g'_b, \eta'_b)}(\fr{s}_{\delta - e}),\ 
\text{since $K^{*}_{\tilde{X}} = -e$;}
\]
we finally get:
\[
\fsw_{(g'_b, \eta'_b)}(\fr{s}_{2\,e - \delta}) = 
\fsw_{(g_b, \eta_b)}(\fr{s}_{a - \delta}).
\]
This completes the proof. This proof is similar to that of Lemma 3 in \cite{Sm-2}, but it uses negative inflation to establish the last equality.
\qed
\end{proof}
\smallskip%

It follows that $q_{\delta-e}$ gives a homomorphism:
\[
q_{\delta-e} \colon \pi_1 (S_{\kappa, \lambda}, \omega)/\pi_1 \diff(\tilde{X}) \cong K(\tilde{X},\omega_{\lambda}) \to \zz_2.
\]
Let $\Delta_{\kappa,\lambda}$ be the (possibly infinite) set of all $\delta$'s that satisfy \eqref{a's}. We may define a homomorphism
\[
q \colon K(\tilde{X},\omega_{\lambda}) 
\to 
\prod_{\delta \in \Delta_{\kappa,\lambda}} \zz_2\quad 
\text{by setting $q = \oplus_{\delta \in \Delta_{\kappa,\lambda}} q_{\delta - e}$}.
\]

\section{Period domains for complex tori}\label{tori}
Most of the material in this section is found 
in \cite{Shioda, Ent-Verbit-1}. See also \cite{Lange} for an in-depth introduction to the theory of complex tori. We consider a complex torus of dimension $2$
\[
X = \cc^2/L,
\]
$L$ being is a lattice in $\cc^2$. By the Kodaira classification theorem, every complex surface diffeomorphic to the real $4$-torus is also biholomorphic to such a quotient. Let $u^1, u^2, u^3, u^4$ be a basis in $H^1(X;\zz)$. Identifying $H^4(X;\zz)$ with $\zz$ through the natural orientation of $X$ as a complex manifold, we have 
\[
u^1 \cup u^2 \cup u^3 \cup u^4 = \text{$1$ or $-1$.}
\]
In the former case, the basis $u^i$ is said to be admissible. We identify the cohomology 
group $H^2(X;\zz)$ with the exterior square $\Lambda^2(L^{*})$ by sending an element $u^{i} \wedge u^{j}$ to the cohomology class 
$u^{i} \cup u^{j}$. The cup product pairing 
$H^2(X;\zz) \times H^2(X;\zz) \to \zz$ makes $H^2(X;\zz)$ into a Euclidean lattice, whose intersection form is even and has signature $(3,3)$. Letting $z_1,z_2$ be the coordinates on $\cc^2$, the differentials $\text{d}\,z_1, \text{d}\,z_2$ is a basis of holomorphic $1$-forms and 
$\phi_{X} = \text{d}\,z_1 \wedge \text{d}\,z_2$ is a holomorphic $2$-form on $X$. Such a holomorphic $2$-from is unique up to a scalar multiple. The class $[\phi_X] \in H^2(X;\cc)$ spans the subspace $H^{2,0}(X)$ of $H^2(X;\cc)$ and satisfies the relations:
\[
\langle [\phi_X], [\phi_X] \rangle = 0,\quad 
\langle [\phi_X], [\overline{\phi_X}] \rangle = 0.
\]
Every complex torus is K{\"a}hler. In fact, the K{\"a}hler cone $\calc(X)$ of $X$ is exactly one of the two connected components of the following set:
\[
\left\{ \kappa \in H^{1,1}(X;\rr)\ |\ 
\langle \kappa, \kappa \rangle > 0 \right\}.
\]
Indeed, we have a one-to-one correspondence between 
the real $(1,1)$-classes on $X$ and the real-valued $(1,1)$-forms on $\cc^2$ with constant coefficients. If $\omega$ is a $(1,1)$-form with constant coefficients, then it descends to a K{\"a}hler form on $X$ iff it tames the complex structure; i.e. $\omega(v, Jv) > 0$ for all $v \neq 0$. Together, all such forms sweep out exactly one of the two connected components of the set of all non-degenerate $(1,1)$-forms with constant coefficients. The other connected component is the set of all $(1,1)$-forms taming $-J$.
\medskip%

Given another complex torus $Y$ and an isomorphism $f^{*} \colon H^1(Y;\zz) \to H^1(X;\zz)$, we call $f^{*}$ an admissible isomorphism if it takes an admissible basis of $H^1(Y;\zz)$ to an admissible basis of $H^1(X;\zz)$. 
Because of the canonical isomorphism $H^2(X;\zz) = \Lambda^2 (H^1(X;\zz))$, any isomorphism $f^{*} \colon H^1(Y;\zz) \to H^1(X;\zz)$ induces an isomorphism:
\[
~^{\wedge}\!f^{*} \colon H^2(Y;\zz) \to H^2(X;\zz).
\]
Shioda \cite{Shioda} proved the Torelli theorem for complex $2$-tori. It is convenient for us to formulate his result as follows:
\begin{theorem}[Shioda]
Let $X$ and $Y$ be two complex tori of dimension 2, and assume that 
there is an admissible isomorphism $f^{*} \colon H^1(Y;\zz) \to H^1(X;\zz)$. In order that $f^{*}$ is induced by a biholomorphism 
$f \colon X \to Y$ it is necessary and sufficient that $~^{\wedge}\!f^{*}$ takes $H^{2,0}(Y)$ to $H^{2,0}(X)$.
\end{theorem}
\medskip%

Let $E = \Lambda^2(L^{*})$ and $E_{\rr} = E \otimes \rr$, $E_{\cc} = E \otimes \cc$. Recalling the canonical identification 
\begin{equation}\label{identify-h}
E = H^2(X;\zz),\quad E_{\cc} = H^2(X;\cc),
\end{equation}
and letting $\langle \phantom{\cdot}, \phantom{\cdot} \rangle$ denote the cup product pairing on $E$, we introduce the period domain:
\[
\Phi= \left\{ \varphi \in E_{\cc}\ |\ 
\langle \varphi, \varphi \rangle = 0,\ \langle \varphi, \bar{\varphi} \rangle > 0 \right\}/\cc^{*} \subset \pp^{5},
\]
which is an open domain in a nonsingular quadric of dimension $4$. 
Under \eqref{identify-h}, the class $[\phi_X] \in H^2(X;\cc)$ gives a point in $\Phi$, called the period of $X$. More generally, if $X$ varies in a complex-analytic family
\[
p \colon \calx \to S,\quad X_{s} = p^{-1}(s),\ s \in S\quad\text{with $X = X_{s_0}$},
\]
and there is given an isomorphism $f^{*} \colon H^1(X_s;\zz) \to H^1(X_{s_0};\zz)$, then, by using 
the Hodge decomposition for each $s \in S$,
\[
H^2(X_s;\cc) = H^{2,0}(X_s;\cc) + 
H^{1,1}(X_s;\cc) + H^{0,2}(X_s;\cc),
\]
and the natural isomorphisms
\[
~^{\wedge}\!f^{*} \colon H^2(X_s;\cc) \cong H^2(X_{s_0};\cc) = E_{\cc},
\]
we define a period mapping 
$\tau_{S} \colon S \to \Phi$ by setting
\[
\tau_{S}(s) = ~^{\wedge}\!f^{*} \left( H^{2,0}(X_s) \right).
\]
Let now $L$ be an abstract abelian group of rank $4$, 
and let a basis $e_1, e_2, e_3, e_4$ in $L$ be fixed. Let $E$, $E_{\rr}$, and $E_{\cc}$ be as before. Let $R \subset \text{Hom}(L,\cc^2)$ be the (open) subset of all group homomorphisms $r \colon L \to \cc^2$ such that $\cc^2/r(L)$ is compact and such 
that $r(e_1),r(e_2),r(e_3),r(e_4)$ gives the positive orientation. Define 
\[
\calm = R/\sim,\quad \text{where $r \sim r'$ iff $r = A \circ r'$ for some $A \in \Gl(2,\cc)$};
\]
\[
\calx =  R \times \cc^2/\sim,\quad \text{where 
$(r,z) \sim (r',z')$ iff $r = A \circ r'$, $z = 
A\left(z' + r'(\ell)\right)$ for some $A \in \Gl(2,\cc)$, $\ell \in L$}.
\]
Forgetting $z$ gives a holomorphic fiber bundle $p \colon \calx \to \calm$ with fiber a complex $2$-torus. Each fiber 
$X_{b} = p^{-1}(b)$, $b \in \calm$ carries a canonical marking; the marking of $H^1(X_b;\zz)$ being given by the basis dual to $r(e_i)$. Then, using the natural isomorphism
\[
j^{*} \colon H^1(X_b;\zz) \cong L^{*},\quad 
~^{\wedge}\!j^{*} \colon H^2(X_b;\cc) \cong E_{\cc};
\]
we can globally define a period mapping: 
\[
\tau_{\calm} \colon \calm \to \Phi\quad \text{by setting $\tau_{\calm}(b) = ~^{\wedge}\!j^{*}\left( H^{2,0}(X_b) \right)$ for all $b \in \calm$.}
\]
In \cite{Shioda}, Shioda proves the following result:
\begin{theorem}[Shioda]
$\tau_{\calm} \colon \calm \to \Phi$ is a biholomorphism.
\end{theorem}
Thus we may identify $\calm$ and $\Phi$. For abbreviation, we use $\phi_b$ for $\tau_{\calm}(b)$. One calls $\calm$ a fine moduli space of complex tori with prescribed marking. Observe that $\calx \to \calm$ has a canonical section, namely the zero-section. This will be useful later. 
\medskip%

Fix a class $\kappa \in E_{\rr}$ with 
$\langle \kappa, \kappa \rangle > 0$. 
Let $\calm_{\kappa} = \left\{ b \in \calm\ |\ \langle \kappa, \phi_b \rangle = 0 \right\}$. As is well known, $\calm_{\kappa}$ has two connected components, each being contractible. One of them, say 
$\calm_{\kappa}^{+}$, has the property that for each $b \in \calm_{\kappa}^{+}$, the class 
$\kappa \in H^2(X_b;\rr)$ is a K{\"a}hler class. 
(Here we identify $H^2(X_b;\rr)$ and $E_{\rr}$ 
via $~^{\wedge}\!j^{*}$.) One calls $\calm_{\kappa}^{+}$ a fine moduli space of $(\kappa-)$polarized complex tori with prescribed marking. In what follows, we drop the superscript in $\calm_{\kappa}^{+}$ and call it $\calm_{\kappa}$, forgetting about the other connected component.

\section{Generic tori} The following material is found in \cite{Lat-McD-Schl, Ent-Verbit-1}. A complex $2$-torus $X$ is said to generic if 
\[
\langle [\phi_X], a \rangle \neq 0\quad 
\text{for all $a \in H^2(X;\zz)$, $a \neq 0$.}
\]
The term \say{generic} is justified 
by the following fact. Let $\calm^{gen} \subset \calm$ be defined as follows:
\[
\calm^{gen} = \left\{ b \in \calm\ |\ 
\langle \phi_b, a \rangle \neq 0\ \text{for all $a \in E$, $a \neq 0$}
\right\}.
\]
Then for each $b \in \calm^{gen}$, $X_b$ is generic. Since $\calm^{gen}$ is a complement to a countable union of proper algebraic subvarieties, it is an open and dense subset of $\calm$.
\medskip%

We recall the following result of Demailly and Paun, which we state for the case of surfaces. See Theorem 0.1 in \cite{Dem-Paun} for the full result.
\begin{theorem}[Demailly-Paun, \cite{Dem-Paun}]\label{dem-pau}
Let $X$ be a compact K{\"a}hler surface. The K{\"a}hler cone of $\calc(X)$ of $X$ is one of the connected components of the set of all real $(1,1)$-classes 
$\kappa$ such that $\langle \kappa, \kappa \rangle > 0$ and such that 
$\langle \kappa, [D] \rangle > 0$ for every effective divisor $D$ on $X$.
\end{theorem} 
\medskip%

Let again $\tilde{X}$ be a blowup of $X$, and let $e \in H^2(\tilde{X};\zz)$ be the class of the exceptional $(-1)$-curve. Letting $\langle e \rangle$ be the subspace of $H^{1,1}(\tilde{X};\rr)$ generated by $e$, we have a natural splitting $H^{1,1}(\tilde{X};\rr) = H^{1,1}(X;\rr) \oplus \langle e \rangle$. The following statement is found in \cite{Lat-McD-Schl}: 
\begin{lemma}\label{kahler-cone}
If $X$ is generic, then $\calc(\tilde{X})$ is one of the connected component of the following set:
\[
\left\{ \kappa - \lambda\,e 
\in H^{1,1}(\tilde{X};\rr)\ |\ 
\sqrt{\langle \kappa, \kappa \rangle} > \lambda > 0 \right\}.
\]
\end{lemma}
\begin{proof}
This follows upon applying Theorem \ref{dem-pau}. See Proposition 3.3 in \cite{Lat-McD-Schl} and Theorem 8.6 in \cite{Ent-Verbit-1} for more general results. \qed
\end{proof}

\section{Proof of Theorem \ref{t:A}}\label{A:proof} 
The notation are as in \S\,\ref{tori}. Consider the 
pullback of $p \colon \calx \to \calm$, which we also denote by $p$, along the embedding $\calm_{\kappa} \subset \calm$. We set $X_b = p^{-1}(b)$, $b \in \calm_{\kappa}$. Let $Z \subset \calx$ be the zero-section of $p$, and let $\tilde{\calx}$ be the blowup of $\calx$ along $Z$. Let $\tilde{p} \colon \tilde{\calx} \to \calm_{\kappa}$ be the composition of the contraction $\tilde{\calx} \to \calx$ followed by the projection 
$p \colon \calx \to \calm_{\kappa}$. Then $\tilde{p} \colon \tilde{\calx} \to \calm_{\kappa}$ is a holomorphic fiber bundle, and for each $b \in \calm_{\kappa}$, the fiber $\tilde{X}_{b} = \tilde{p}^{-1}(b)$ is then the blowup of $X_b$ at the intersection point of $X_b$ and $Z$. Using the 
isomorphism 
\[
~^{\wedge}\!j^{*} \colon H^2(X_{b};\cc) \cong E_{\cc}
\] 
and the splitting 
\[
H^2(\tilde{X}_{b};\cc) = H^2(X_{b};\cc) \oplus \langle e \rangle,
\]
$e$ being the class dual to the exceptional $(-1)$-curve, we obtain a natural isomorphism 
\[
H^2(\tilde{X}_{b};\cc) \cong E_{\cc} \oplus \langle e \rangle.
\]
Let $\lambda > 0$ be so that 
$\langle \kappa, \kappa \rangle - \lambda^2 > 0$. Consider the class $\kappa - \lambda\,e \in E_{\rr} \oplus \langle e \rangle$. Recall that for all $b \in \calm_{\kappa}$, we have $\kappa - \lambda\,e \in H^{1,1}(\tilde{X}_b;\cc)$. Let $\calm_{\kappa,\lambda} \subset \calm_{\kappa}$ be those points $b \in \calm_{\kappa}$ such that $\kappa - \lambda\,e \in \calc(\tilde{X}_b)$. 
\begin{lemma}\label{gen-dense}
If $\kappa$ is non-resonant, then $\calm_{\kappa,\lambda}$ is open and dense in $\calm_{\kappa}$. Also, $\calm_{\kappa,\lambda}$ is connected.
\end{lemma}
\begin{proof}
Openness follows by the Kodaira-Spencer stability theorem. See \cite[\S\,6]{K-S}. 
Let us show that $\calm_{\kappa,\lambda} \subset \calm_{\kappa}$ is dense. Let again $\calm^{gen}_{\kappa} \subset \calm_{\kappa}$ be defined as
\[
\calm^{gen}_{\kappa} = \left\{ b \in \calm_{\kappa}\ |\ 
\langle \phi_b, a \rangle \neq 0\ \text{for all $a \in E$, $a \neq 0$}
\right\}.
\]
Recall the identification 
$\calm = \Phi \subset \pp^{5}$. 
With this identification made, $\calm_{\kappa}$ is the
hyperplane section of $\Phi$ given by 
the intersection with $\langle \kappa, \phi \rangle = 0$; $\calm^{gen}_{\kappa} \subset \calm_{\kappa}$ is the complement of the union of countably many hyperplanes given by $\langle a, \phi \rangle = 0$, $a \in E$. If $\kappa$ were (a multiple of) a rational class, then $\calm^{gen}_{\kappa}$ would be empty; otherwise, it must be open and dense in $\calm_{\kappa}$. However, if $\kappa$ is non-resonant, then it is not a real multiple of a rational class. Finally, by Lemma \ref{kahler-cone}, we have the inclusion $\calm^{gen}_{\kappa} \subset \calm_{\kappa,\lambda}$. 
\smallskip%

What remains is to show that 
$\calm_{\kappa,\lambda}$ is connected. Again, 
since $\calm_{\kappa}$ is irreducible and $\calm^{gen}_{\kappa} \subset \calm_{\kappa}$ is the complement to the union of countably many complex-analytic sets, it follows that $\calm^{gen}_{\kappa}$ is connected by piecewise smooth analytic arcs. 
Since 
$\calm_{\kappa,\lambda} \subset \calm_{\kappa}$ is open and 
$\calm^{gen}_{\kappa} \subset \calm_{\kappa,\lambda}$ is dense, the connectedness of $\calm_{\kappa,\lambda}$ follows. \qed 
\end{proof}
\smallskip%

Let us choose a smooth-varying family of K{\"a}hler forms $\omega_b \in H^{1,1}(\tilde{X}_b;\rr)$ representing the class $\kappa - \lambda\,e$. (Back in \S\,\ref{main}, we called it $\tilde{\Omega}_{b}$.) Fix a base point $b_0 \in \calm_{\kappa,\lambda}$. We may assume that $b_0 \in \calm^{gen}_{\kappa}$. Since $\calm_{\kappa}$ is contractible, we may choose a $C^{\infty}$-trivialization 
$f \colon \tilde{X}_{b_0} \times \calm_{\kappa} \to \tilde{\calx}$. Then $f^{*} \omega_b$ gives a family of symplectic forms on $\tilde{X}_{b_0}$ 
parametrized by $\calm_{\kappa,\lambda}$, and we obtain a monodromy map:
\begin{equation}\label{moses}
\pi_1(\calm_{\kappa,\lambda}) \to 
K(\tilde{X}_{b_0}, \omega_{b_0}).
\end{equation}
Abusing notation, we also write $\omega_b$ for $f^{*} \omega_b$. 
\medskip%

Let $\Delta_{\kappa,\lambda}$ be defined as follows: 
\[
\Delta_{\kappa,\lambda} = 
\left\{ \delta \in E\ |\ \langle \delta, \delta \rangle = 0,\ 0 < \langle \kappa, \delta \rangle \leq 
\lambda,\ \text{$\delta$ is indivisible} \right\}.
\]
We will use the following basic property of integral skew-symmetric quadratic form:
\begin{lemma}\label{skew-forms}
If $\delta$ is an integral, skew-symmetric quadratic form on $L \cong \zz^n$, 
then there is a basis 
$u_1, v_1, \dots, u_n,v_n$ for $L$ such that if 
$u^{i}, v^{i}$ is dual to $u_i,v_i$, then 
$\delta$ takes the forms:
\[
\delta = d_1\,u^1 \wedge v^1 + \cdots + d_n\,u^n \wedge v^n,\quad d_i \in \zz.
\]
\end{lemma}
\begin{proof}
See Section called \say{The Riemann Conditions} in \cite[Ch.\,6]{G-H}. \qed
\end{proof}
\begin{lemma}\label{g-h-lemma}
If $\kappa$ is non-resonant, then 
$\Delta_{\kappa,\lambda}$ is infinite.
\end{lemma}
\begin{proof}
It suffices to show for any 
$\lambda'$ with $0 < \lambda' < \lambda$ there exists 
$\delta \in E$ such that $\langle \delta, \delta \rangle = 0$, and $0 < \langle \kappa, \delta \rangle < \lambda'$. To begin with, if $\kappa$ is non-resonant, then we do automatically have that $\langle \kappa, \delta \rangle \neq 0$. Pick $\delta_1 \in E$ such that 
$\langle \delta_1, \delta_1 \rangle = 0$. By Lemma \ref{skew-forms}, there exists $\delta_2 \in E$ such that $\langle \delta_2, \delta_2 \rangle = 0$, and also $\langle \delta_1, \delta_2 \rangle = 0$. It follows then that 
\[
\langle p_1 \delta_1 + p_2 \delta_2, p_1 \delta_1 + p_2 \delta_2 \rangle = 0,
\]
and, since $\kappa$ is non-resonant, 
\[
\langle \kappa, p_1 \delta_1 + p_2 \delta_2 \rangle \neq 0\quad \text{for all $p_1,p_2 \in \zz$, unless $p_1 = p_2 = 0$.}
\]
Be given $\lambda' > 0$, we can always find 
integers $p_1$ and $p_2$ such that 
$0 < \langle \kappa, p_1 \delta_1 + p_2 \delta_2 \rangle < \lambda'$. \qed
\end{proof}
\medskip%

Define 
\[
\cald_{\lambda} = \left\{ b \in \calm_{\kappa}\ |\ 
\langle \phi_b, \delta \rangle \neq 0\quad 
\text{for all $\delta \in \Delta_{\kappa,\lambda}$}
\right\}.
\]
Identify, as usual, $\calm_{\kappa}$ with a hyperplane section $\Phi$. For $\delta \in \Delta_{\kappa,\lambda}$, let $H_{\delta}$ be defined as follows:
\[
H_{\delta} = \left\{ \phi \in \Phi\ |\ 
\langle \delta, \phi \rangle = 0,\ \langle \kappa, \phi  \rangle = 0 \right\}.
\]
With this identification made, $\cald_{\lambda} \subset \calm_{\kappa}$ becomes the 
complement to the union of all $H_{\delta}$'s.
\begin{lemma}\label{-1-torus}
$\calm_{\kappa,\lambda} \subset \cald_{\lambda}$.
\end{lemma}
\begin{proof}
To begin with, we show that if $X = \cc^2/L$ is a complex $2$-torus, and $\delta \in H^2(X;\zz)$ is an \emph{indivisible} class such that 
\[
\delta \in H^{1,1}(X;\rr),\quad \langle \delta, \delta \rangle = 0;
\]
then there exists a smooth irreducible elliptic curve $D \subset X$ such that $[D] = \pm \delta$. Here $[D] \in H^2(X;\zz)$ is dual to $D$. Indeed, we may view $\delta$ as an integral (with respect to the lattice $L$), skew-symmetric $(1,1)$-form on $\cc^2$. By Lemma \ref{skew-forms}, there is a basis 
$u_1, v_1, u_2,v_2$ for $L$ such that if 
$u^i, v^i$ is dual to $u_i, v_i$, then $\delta$ takes the form:
\[
\delta = d_1\,u^1 \wedge v^1 + d_2\,u^2 \wedge v^2,\quad d_1, d_2 \in \zz,
\]
and we have that:
\[
\langle \delta, \delta \rangle = 2 d_1 d_2 = 0.
\]
We may assume that $d_2 = 0$. Because $\delta$ is indivisible, we may also assume that $d_1 = 1$. It clear that the subspace $\text{ker}\,\delta \subset \cc^2$ is of real dimension $2$. Also, $\text{ker}\,\delta$ is a complex subspace, since $\delta$ is of type $(1,1)$. But then $\text{ker}\,\delta$ descends to a complex subtorus 
$D$ of $X$. Let us view $D$ as a subgroup of $X$ and set $Y = X/D$. The fibering $\vartheta \colon X \to Y$ (given by the quotient map $X \to X/D$) is holomorphic and has fibers homologous to the torus $D$. Let $\delta_{Y} \in H^2(Y;\zz)$ be defined by 
\[
\int_{Y} \delta_{Y} = 1\quad \text{with respect to the orientation of $Y$ as a complex manifold.}
\]
Then it clear that $[D] = \vartheta^{*} \delta_{Y}$. On the other hand, it is also clear that
\[
\vartheta^{*} \delta_{Y} = \text{$-\delta$ or $+\delta$.}
\]
In the latter case, we also have $\langle \delta , \kappa \rangle > 0$ for any K{\"a}hler class 
$\kappa$ of $X$.
\smallskip%

Observe that if $\tilde{X} \to X$ is the blowup of $X$ at a point, then $[D] - e \in H^{1,1}(\tilde{X};\zz)$ is effective, being represented by a smooth irreducible elliptic curve of self-intersection number $(-1)$.
\smallskip%

Returning to $\cald_{\lambda}$, let us pick 
$b \in \calm_{\kappa}$ and $\delta \in \Delta_{\kappa, \lambda}$ so that $\langle \varphi_b, \delta \rangle = 0$. As is shown above, either $\delta - e$ or 
$-\delta - e$ is effective in $\tilde{X}_b$; in fact, it must be $\delta - e$ that is effective, because we have $\langle \kappa, \delta \rangle > 0$. But then $\kappa - \lambda\, e$ is not a K{\"a}hler class of $\tilde{X}_b$, because we also have
\[
\langle \kappa - \lambda\,e, \delta - e \rangle = 
\langle \kappa, \delta \rangle - \lambda \leq 0.
\]
Hence, $b \not\in \calm_{\kappa, \lambda}$. \qed
\end{proof}
\medskip%

Even though 
$\Delta_{\kappa,\lambda}$ is infinite, we can show that $\calm_{\kappa}$ is a manifold; namely, let us prove:
\begin{lemma}\label{finite-h-delta}
Each $\phi \in \calm_{\kappa}$ has a neighbourhood $U$ such that $U \cap H_{\delta} = \emptyset$ for all but finitely many $\delta \in \Delta_{\kappa,\lambda}$. In particular, $\cald_{\lambda}$ is an open subset of $\calm_{\kappa}$.
\end{lemma}
\begin{proof}
The proof is similar to that of Lemma 5 in \cite{Sm-2}, but we give a sketch for the sake of completeness. To begin with, we set $x_1 = \text{Re}\, \phi$, $x_2 = \text{Im}\, \phi$. Fix some euclidean norm $||\hphantom{x}||$ on $E_{\rr}$. It is clear that any 
$||\hphantom{x}||$-ball contains only finitely many 
elements of $\Delta_{\kappa,\lambda}$. Suppose, 
contrary to our claim, that there is an unbounded sequence $\left\{ \delta_i \right\}_{k=1}^{\infty}$ such that
\begin{center}
0 < $\langle \delta_i, \kappa \rangle \leq \lambda$, and also that $||\delta_i|| \to \infty$ and $\left( \delta_i, x_1 \right), \left( \delta_i, x_2 \right) \to 0
$ as $i \to \infty$.    
\end{center}
Assuming, as we may, that 
$\left\{ \left\{ \delta_i \right\}/||\delta_i|| \right\}_{i=1}^{\infty} 
\to \delta_{\infty} \in E_{\rr}$ as $i \to \infty$, we obtain four pairwise orthogonal non-zero vectors 
$(\delta_{\infty},\kappa,x_1,x_2)$ such that
\[
\delta_{\infty}^{2} = 0\quad\text{and}\quad\kappa^2 > 0,\quad x_1^2 > 0,\quad x_2^2 > 0.
\]
Such a configuration of vectors, however, is not realizable in the space of signature $(3,3)$. \qed
\end{proof}
\medskip%

Having established the open and dense inclusions $
\calm_{\kappa,\lambda} \subset \cald_{\lambda} \subset \calm_{\kappa},
$ we claim:
\begin{lemma}\label{surjective}
The inclusion induced homomorphism 
$\pi_1(\calm_{\kappa,\lambda}, b_0) \to 
\pi_1(\cald_{\lambda},b_0)$ is surjective.
\end{lemma}
\begin{proof}
We assumed earlier that $b_0 \in \calm^{gen}_{\kappa}$. If $\gamma$ is a loop in $\cald_{\lambda}$ based at $b_0$, then it can be $C^{0}$-approximated by a based loop lying 
entirely in $\calm^{gen}_{\kappa}$. But 
$\calm^{gen}_{\kappa} \subset \calm_{\kappa,\lambda}$, which proves the lemma. \qed
\end{proof}
\medskip%

Let $\gamma$ be a based loop in $\cald_{\lambda}$. Since $\calm_{\kappa}$ is contractible, we may find a nullhomotopy $g \colon D \to \calm_{\kappa}$ of $\gamma$. It follows from 
Lemma \ref{finite-h-delta}, $g(D)$ intersects only finitely many varieties $H_{\delta}$. Perturbing 
$g$ if needed, we may assume that it is transverse to each $H_{\delta}$; we set:
\[
\ell_{\delta}(\gamma) = \text{Cardinality}\left(
g(D) \cap H_{\delta}  \right)\,\text{mod}\,2.
\]
It is clear that $\ell_{\delta}(\gamma)$ does not depend on a particular choice of $\gamma$ within a (free) homotopy class. It follows that for each $\delta \in \Delta_{\kappa,\lambda}$, the map 
$\gamma \to \ell_{\delta}(\gamma)$ gives a group homomorphism $\pi_1(\cald_{\lambda}) \to \zz_2$. Summing over all $\delta \in \Delta_{\kappa,\lambda}$, this gives a group homomorphism: 
\[
\ell = \oplus_{\delta \in \Delta_{\kappa,\lambda}} \ell_{\delta} \colon \pi_1(\cald_{\lambda}) \to 
\bigoplus_{\delta \in \Delta_{\kappa,\lambda}} \zz_2.
\]
For each $\delta_0 \in \Delta_{\kappa,\lambda}$, we choose a based loop $\gamma_{\delta_0}$ satisfying the following condition: There exists a nullhomotopy $g \colon D \to \calm_{\kappa}$ such that $g$ is transverse to $H_{\delta_0}$, $g(D) \cap H_{\delta_0}$ consists of a single point, and $g(D) \cap H_{\delta}$ is empty for all 
$\delta \in \Delta_{\kappa,\lambda} - \left\{ \delta_0 \right\}$.   
\begin{lemma}\label{l:fund-group}
$\pi_1(\cald_{\lambda},b_0)$ is normally-generated by the set 
$\left\{ \gamma_{\delta} \right\}_{\delta \in \Delta_{\kappa,\lambda}}$.
\end{lemma}
\begin{proof}
The proof is analogous to that of Lemma 7 in \cite{Sm-2}. \qed
\end{proof}
\medskip%

To prove Theorem \ref{t:A}, we show that the following diagram is commutative:
\[
\begin{tikzcd}
\pi_1(\calm_{\kappa,\lambda}, b_0) \arrow{d}[swap]{\text{surjective}} \arrow{r}{\eqref{moses}} & K(\tilde{X}_{b_0}, \tilde{\Omega}_{b_0}) \arrow{d}{q}\\
\pi_1(\cald_{\lambda}, b_0) \arrow{r}{\ell} & \bigoplus_{\delta \in \Delta_{\kappa,\lambda}}\zz_2\,.
\end{tikzcd}
\]
To begin with, observe that we may assume 
that each $\gamma_{\delta}$ lies entirely in $\calm_{\kappa,\lambda}$; this follows from Lemma \ref{surjective}. Fix $\delta_0 \in \Delta_{\kappa,\lambda}$. Let us view $\gamma_{\delta_0}$ as a loop of symplectic forms on $\tilde{X}_{b_0}$. Then it suffices to show that 
\[
q\left(\gamma_{\delta_0}\right) = 
\begin{cases}
      1 & \text{for $\delta = \delta_0$,} \\
      0 & \text{for all $\delta = \Delta_{\kappa,\lambda} - \left\{ \delta_0 \right\}$.}
\end{cases}
\]
Let us view $q$ as map from the free loop space of $\calm_{\kappa,\lambda}$ into $\zz_2$.
\begin{lemma}\label{free}
Let $\gamma$ and $\gamma'$ be loops in $\calm_{\kappa,\lambda}$. Suppose that they are homotopic in $\cald_{\lambda}$ as free loops; then $q(\gamma) = q(\gamma')$.
\end{lemma}
\begin{proof}
Let us indeed assume that $\gamma$ and $\gamma'$ co-bound an annulus $A$ in $\cald_{\lambda}$. 
For each $b \in A$, let us choose a family of K{\"a}hler $\Omega_b \in H^{1,1}(\tilde{X}_b;\rr)$ that agrees with $\omega_b$ on the boundary $\del A = \gamma \cup \gamma'$, and satisfies: 
\[
[\Omega_b] = \kappa - h\,e \in H^2(X_b;\rr)\quad \text{for some function $h = h(b)$.}
\]
We may choose $h$ so that $h(b) \leq \lambda$ for all $b \in A$. 
Since $\Omega_b$ and $\omega_b$ agree on $\del A$, we have $h(b) = \lambda$ for all $b \in \del A$. 
\smallskip%

Let $g_b$ be the family of 
fiberwise Hermitian metrics on $\tilde{X}_b$ associated to $\Omega_b$. For each $\varepsilon \in E$, let $\fr{s}_{\varepsilon}$ be the spin$^{\cc}$ 
structure on $T_{\tilde{\calx}/A}$ which, when restricted 
to $\tilde{X}_b$, satisfies $c_1(\fr{s_{\delta}}) = c_1(X_b) + 2\, \varepsilon$. As usual, we set:
\[
\eta_{b} = -i F_{{A_0}_t}^{+} - \rho\, \Omega_b.
\]
For each $b \in A$, we have that
\[
-\rho \langle [\Omega_b], \kappa - \lambda\,e \rangle = 
-\rho ( \langle \kappa, \kappa \rangle - \lambda\, h) < 0\quad\text{for all $\rho > 0$,}
\]
so that \eqref{admin} holds for $\rho > 0$ large enough. 
\smallskip%

Put $\varepsilon = \delta - e$ for some $\delta \in \Delta_{\kappa, \lambda}$. Consider the Seiberg-Witten equations parametrized by the family 
$(g_b,\eta_b)$, $b \in A$ and the spin$^{\cc}$ structure $\fr{s}_{\delta - e}$. Observe that for each $b \in A$, these equations 
have no solutions. Indeed, if $b \in \cald_{\lambda}$, then $\delta - e \in H^2(\tilde{X}_b;\zz)$ is not a $(1,1)$-class, so Theorem \ref{t:kahler} applies. It follows then that
\[
Q^{+}_{\delta - e}(\gamma) = Q^{+}_{\delta - e}(\gamma').
\]
Put $\varepsilon = 2\,e - \delta$ for some $\delta \in \Delta_{\kappa, \lambda}$. Choose $\mu > 0$ so that \eqref{mu-close-eps} holds; i.e., so that to satisfy
\[
\langle 2\,e - \delta, \kappa - (\lambda - \mu)\,e \rangle \leq 0.
\]
Deform $\Omega_b$ by making $h$ so that $h \leq \lambda - \mu$ 
for all $b \in A$. It follows then that for each $b \in \del A$,
\[
\Omega_b = \alpha_{\mu}(\omega_b),
\]
where $\alpha_{\mu}$ is as in Theorem \ref{mcduff-alpha}. 
Let again $(g_b,\eta_b)$ be the 
parameters defined for the (deformed) family $\Omega_b$. For each $b \in A$, the Seiberg-Witten equations for $(g_b,\eta_b)$ and $\fr{s}_{2\,e - \delta}$ have no solutions. This again follows upon 
applying Theorem \ref{t:kahler}. But then we have:
\[
Q^{-}_{\delta - e}(\gamma) = Q^{-}_{\delta - e}(\gamma'),
\]
and hence 
$q_{\delta - e}(\gamma) = q_{\delta - e}(\gamma')$. This completes the proof. \qed
\end{proof}
\medskip%

We apply Lemma \ref{free} as follows. Let us view $\gamma_{\delta_0}$ as a free loop in $\cald_{\lambda}$. Make a homotopy of $\gamma_{\delta_0}$ into a loop so small that it becomes a boundary of a holomorphic disc $D$ transverse to $H_{\delta_0}$. Perturbing $D$ if needed, we may assume that it intersects $H_{\delta_0}$ at a good point; that is, letting 
$b_0$ be that point, we should have $\langle \delta, \phi_{b_0} \rangle \neq 0$ for all $\delta \in \Delta_{\kappa,\lambda} - \left\{ \delta_0 \right\}$. By Lemma \ref{finite-h-delta}, $D$ can be chosen 
small enough so that:
\[
D \cap H_{\delta_{0}} = \left\{ \phi_{b_0} \right\}\quad 
\text{and}\quad 
D \cap H_{\delta} = \emptyset\ \text{for each $\delta \in \Delta_{\kappa,\lambda} - \left\{ \delta_0 \right\}$.}
\]
Finally, may assume that 
$\gamma_{\delta_0} \subset \calm_{\kappa,\lambda}$. Indeed, as in Lemma \ref{surjective}, we can choose a $C^{0}$-perturbation $\gamma'_{\delta_0}$ of $\gamma_{\delta_0}$ such that 
$\gamma'_{\delta_0} \subset D$, and such that $\gamma'_{\delta_0} \subset \calm_{\kappa,\lambda}$, Shrinking $D$, we can assume that $\del D = \gamma'_{\delta_0}$. 
\medskip%

The rest of the proof follows \cite[\S\,8]{Sm-2}; we repeat the key steps here for the sake of completeness. Let again $\Omega_b$, $b \in D$ be a family of K{\"a}hler form on $\tilde{X}_b$ that agrees with $\omega_b$ on $\del D$, and let $(g_b,\eta_b)$ be the associated family of metrics and perturbations. 
\begin{lemma}
$q_{\delta - e}(\gamma_{\delta_0}) = 0$ for 
each $\delta \in \Delta_{\kappa,\lambda} - \left\{ \delta_0 \right\}$.
\end{lemma}
\begin{proof}
If $\delta \neq \delta_0$, then 
$\delta \centernot\in H^{1,1}(\tilde{X}_b;\rr)$ for all $b \in D$, and we may again apply 
Theorem \ref{t:kahler}. The proof proceeds with the calculation done in Lemma \ref{free}; although in \ref{free} we worked on the annulus $A$, the calculation can be extended to arbitrary compact subset of $\calm_{\kappa}$. \qed  
\end{proof}
\begin{lemma}
$q_{\delta_0 - e}(\gamma_{\delta_0}) = 1$.
\end{lemma}
\begin{proof}
If $b \in D - \left\{ b_0 \right\}$, then 
$\delta_0 - e \centernot\in H^{1,1}(\tilde{X}_b;\rr)$. On the other hand, we have 
$\delta_0 - e \in H^{1,1}(\tilde{X}_{b_0};\rr)$. Indeed, 
$\delta_0 - e \in H^{1,1}(\tilde{X}_b;\rr)$ if and only if 
$\delta_0 \in H^{1,1}(X_b;\rr)$. But 
$\delta_0 \in H^{1,1}(X_b;\rr)$ if and only if 
$\langle \delta_0, \phi_{b} \rangle = 0$. This is equivalent to $\phi_{b} \in D \cap H_{\delta_0}$. But $D$ is chosen in such a way that $D \cap H_{\delta_0} = \left\{ \phi_{b_0} \right\}$.
\smallskip%

We saw in Lemma \ref{-1-torus} that $\delta_0 - e$ is represented by a smooth irreducible elliptic curve of self-intersection $(-1)$. By positivity of intersections, no other holomorphic curve represeting $\delta_0 - e$ can exist. Thus, by Theorem \ref{t:kahler}, the moduli space 
$\mathfrak{M}^{ \fr{s}_{\delta_{0} - e} }_{(g_b,\eta_b)}$ consists of a single point, which lies over $b_0 \in D$. In \cite[\S\,6]{Sm-1}, a sufficient condition is found that guarantees that this point is tranversally cut out; namely, if consider the infinitesimal variation of Hodge structures (\cite{Griff}):
\[
\Omega_{*} \colon T_{D} \to \text{Hom}\,(H^{1,1}, H^{0,2}),\quad\text{where $H^{p,q} = H^{p,q}(\tilde{X}_{b_0};\cc)$.} 
\]
then $\mathfrak{M}^{ \fr{s}_{\delta_{0} - e} }_{g_b,\eta_b}$ is tranversally cut out, provided 
\[
\delta_0 \centernot\in \text{ker}\, \Omega_{*}(\del_t),\quad\text{where 
$\del_t$ is a generator for $T_{D}$.} 
\]
But this is precisely the condition that $D$ is transverse to $H_{\delta_0}$, a condition which is indeed met. It follows then that $Q^{+}_{\delta_0 - e}(\gamma_{\delta_0}) = 1$.  
\smallskip%

Let us deform $\Omega_b$ in exactly the same way as in Lemma \ref{free}. We then have:
\[
\langle [\Omega_{b}], 2\,e - \delta \rangle \leq 0.
\]
Then, using Theorem \ref{thm:taubes}, we get $Q^{-}_{\delta_0 - e}(\gamma_{\delta_0}) = 0$, and hence also $q_{\delta_{0} - e}(\gamma_{\delta_0}) = 1$. \qed
\end{proof}

\bibliographystyle{plain}
\bibliography{references}

\end{document}